\documentclass[a4paper,12pt]{article}
\usepackage{enumerate}
\usepackage[shortlabels]{enumitem}
\usepackage{graphicx}
\usepackage{epstopdf}
\usepackage{amsmath,amsthm}
\newtheorem{theorem}{Theorem}
\newtheorem{conjecture}{Conjecture}
\newtheorem{proposition}{Proposition}
\usepackage{amssymb}
\usepackage{wasysym}
\usepackage{sectsty}
\usepackage[hidelinks]{hyperref}
\usepackage{url}
\usepackage{xspace, float}
\usepackage{algorithmic} 
\usepackage{algorithm}
\usepackage{color}
\usepackage{parskip}

\usepackage[margin=2.54cm]{geometry}

\newcommand{\sz}{wSz}

\graphicspath{{./figures/}}
\DeclareGraphicsExtensions{.pdf, .png, .jpg}

\usepackage[margin=1cm,%
            font=small,%
            format=hang,%
            labelsep=period,%
            labelfont=bf]{caption}
\usepackage{cite, longtable, array, booktabs, dashrule}

\begin{document}

\title{
\LARGE\bf 
On Extremal Graphs \\[3mm] of Weighted Szeged Index
}

\author{\bf Jan Bok$^a$,
    Boris Furtula$^{b,}$\footnote{Corresponding author.}\,\,,
    Nikola Jedli\v{c}kov\'{a}$^a$,\\
    \bf Riste \v{S}krekovski$^d$\\[3mm]
\normalsize\it $^a$Computer Science Institute \& Department of Applied
Mathematics,\\ 
\normalsize\it Faculty of Mathematics and Physics,\\
\normalsize\it Charles University, Malostransk\'{e} n\'{a}m\v{e}st\'{i} 25,
11800, Prague, Czech Republic\\
\normalsize {\tt bok@iuuk.mff.cuni.cz\,, jedlickova@kam.mff.cuni.cz}\\[2mm]
\normalsize\it $^b$Faculty of Science, University of Kragujevac,\\
\normalsize\it P.O.Box 60, 34000 Kragujevac, Serbia\\
\normalsize {\tt boris.furtula@pmf.kg.ac.rs}\\[2mm]
%
%
\normalsize\it $^d$ Faculty of Information Studies, 8000 Novo Mesto,\\
\normalsize\it \& Faculty of Mathematics and Physics, University of Ljubljana, 
1000 Ljubljana, \\
\normalsize\it \&  FAMNIT, University of Primorska, 6000 Koper, Slovenia \\
\normalsize {\tt skrekovski@gmail.com}
}

\date{}

\maketitle

\thispagestyle{empty}

\begin{abstract}
An extension of the well-known Szeged index was introduced recently, named as
{\it weighted Szeged index ($\sz(G)$)\/}. This paper is devoted to characterizing
the extremal trees and graphs of this new topological invariant. In particular,
we proved that the star is a tree having the maximal $\sz(G)$\,. Finding a tree
with the minimal $\sz(G)$ is not an easy task to be done. Here, we present the
minimal trees up to 25 vertices obtained by computer and describe the regularities
which retain in them. Our preliminary computer tests suggest that a tree
with the minimal $\sz(G)$ is also the connected graph of the given order that
attains the minimal weighted Szeged index. Additionally, it is proven that among
the bipartite connected graphs the complete balanced bipartite graph
$K_{\left\lfloor n/2\right\rfloor\left\lceil n/2 \right\rceil}$ attains the maximal
$\sz(G)$\,. We believe that the $K_{\left\lfloor n/2\right\rfloor\left\lceil n/2 
\right\rceil}$ is a connected graph of given order that attains the maximum 
$\sz(G)$\,.


%
%
%

\end{abstract}

\section{Introduction}


An idea that the structure of a molecule governs its physico--chemical properties
is as old as modern chemistry. The problem of quantifying chemical structures was
overcome by introducing the so-called molecular descriptors, which are the carriers
of molecular structure information. These structural invariants are usually numbers
or set of numbers. Nowadays, there are thousands of molecular descriptors that are
commonly used in various disciplines of chemistry \cite{ToCo09}.

Topological indices are significant class among the molecular descriptors. 
The oldest and one of the most investigated topological indices is Wiener index,
introduced in 1947 \cite{Wien47}. Although its definition is based on distances
($d(x,y)$) among pairs of vertices $(x,y)$ in a graph $G$\,, in the seminal paper
\cite{Wien47}, Wiener used the following formula for its calculation: 
\begin{equation}\label{wiener}
  W(T) = \sum_{e=\{uv\}\in E(T)} n_u(e)\cdot n_v(e)
\end{equation}
where $n_u(e)$ is cardinality of the set $N_u(e=\{uv\}) = \{ x \in V(G): d(x,u)
< d(x,v) \}$\,.

Formula \eqref{wiener} is valid only for acyclic molecules that are in chemical
graph theory represented by trees. This fact triggered the introduction of a
novel topological index, which definition coincides with Eq. \eqref{wiener}, but
its scope of usability is widen to all simple connected graphs. The name of this
invariant is Szeged index, $Sz(G)$ \cite{Gutm94} and it is defined as
\begin{equation}\label{szeged}
Sz(G) = \sum_{e=\{uv\}\in E(G)} n_u(e)\cdot n_v(e)
\end{equation}
Shortly after its introduction, Szeged index was attracted much attention in the
mathematical chemistry community. Nowadays there is a vast literature presenting
scientific researches deeply related to the Szeged index (e.g. for some recent
results see \cite{Dobr18, JiHo18, Trat17} and references cited therein).

Inspired by an extension of the Wiener index, today known as degree distance
\cite{DoKo94, Gutm94DD}, Ili\'c and Milosavljevi\'c proposed a modification
of the Szeged index \cite{ilic} as a topological invariant that is worthy of
investigations. This topological descriptor is named as \textit{weighted
Szeged index} and is defined as follows:
\begin{equation}\label{wszeged}
\sz(G) = \sum_{e=\{uv\} \in E(G)} [\deg(u)+\deg(v)]\cdot n_u(e)\cdot n_v(e)
\end{equation}
where $\deg(u)$ is degree of the vertex $u$\,.

Half of decade has been passed since this invariant was introduced, but till
today only some mediocre research activities involving $\sz(G)$ \cite{NaPa14, %
PaKa16, PaKa18} were performed. Therefore, this paper contributes to the 
answering the elemental questions that are arising when someone works with
a novel topological descriptor. In particular, the next section is reserved
for the results and conjectures on graphs that are maximizing the $\sz(G)$\,,
and the other one to the characterization of graphs that are minimizing it. 

In the paper we will use the following definitions. By the \emph{star graph}
$S_n$ we mean the complete bipartite graph $K_{1,n-1}$. We say that some vertex
is an \emph{internal leaf} of a graph $G$ if this vertex becomes a leaf after
removing all leaves of $G$. The subgraph of $G$ induced by a set of vertices
$B\subseteq V(G)$ is denoted by $G[B]$.

\section{On graphs having maximum weighted Szeged\newline index}

The simplest connected graphs are trees and they are usually the first
class of graphs on which researchers are testing various properties and
behaviors of novel topological descriptors. Thence, we are starting this
section on graphs that are maximizing weighted Szeged index by characterizing
a tree having maximum $\sz(G)$\,. 

\vspace{3mm}

\begin{theorem} \label{thm:maximum}
	A tree with the maximum weighted Szeged index among $n$-vertex trees
	is the star graph $S_n$.
\end{theorem}
\begin{proof}
	Contrary to the statement in the Theorem \ref{thm:maximum}, let us assume 
	that there exists an extremal tree $T$ with $n$ vertices which is not a
	star. Since $T$ is a non-star tree, there exists an edge $uv$ of $T$
	such that $u$ is an internal leaf and $v$ is a non-leaf vertex of $T$\,.
	For the sake of simplicity, let us label $\deg(u)$ with $a$ and $\deg(v)$
	with $b$\,. Let $x_1, x_2, \ldots, x_{a-1}$ be the leaves attached at $u$.
	Also, let	$y_1, y_2, \ldots, y_{b-1}$ be the neighbors of $v$ distinct from
	$u$. We denote by $Y_i$ the set of vertices of the $T - \{vy_i\}$ containing
	the vertex $y_i$. Without loss of generality, it is assumed that $Y_1$ is a
	biggest such set, i.e. $|Y_1| \ge |Y_i|$ for every $i > 1$.
	
	Let us denote by $A$ the set of vertices of the $T-v$ containing the vertex
	$u$ and	by $B = V(T) \setminus A$. Note that $|A| + |B| = n$, $|A|=a$, and
	$|A|\,,\,|B|\geq 2$.
	
	It will be demonstrated that the transformation of a tree $T$ into $T'$,
	shown in Figure \ref{fig:contraction} (a contraction of the edge $uv$
	in the tree $T$ into a leaf $w$ in the tree $T'$\,), will always raise
	the $\sz(G)$\,, i.e. $\Delta := \sz(T') - \sz(T) > 0$.
	\begin{figure}[H]
		\centering
		\includegraphics[width=.8\linewidth]{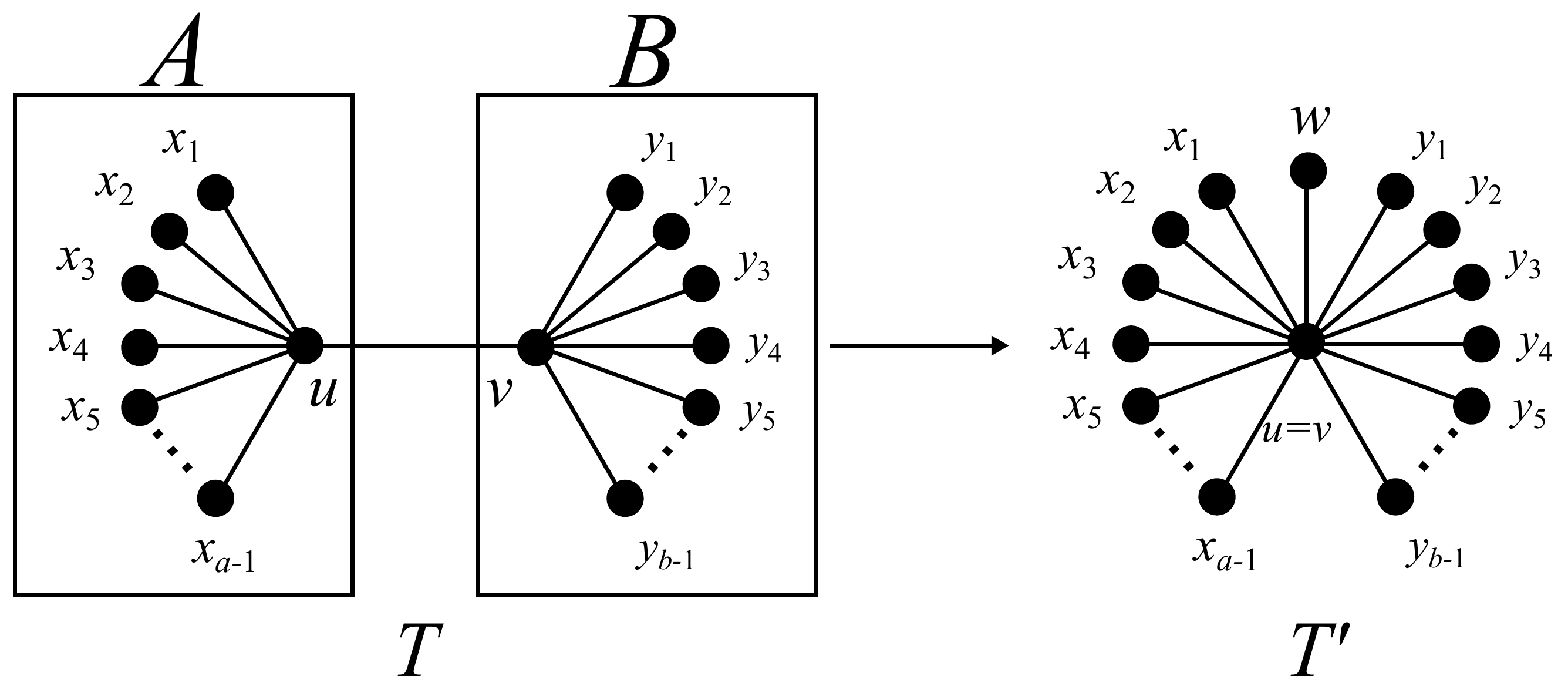}
		\caption{}
		\label{fig:contraction}
	\end{figure}
	
	Note that the contribution of edges in $T[B-v]$ stays the same in $\sz(T')$
	and $\sz(T)$. Thus, we can write
	$$
		\Delta = \sum_{i=1}^{a-1} (b-1)n_{x_i}(x_iu)n_u(x_iu) 
					 + \sum_{i=1}^{b-1} (a-1)n_{y_i}(y_iv)n_v(y_iv) 
					 + (a+b)(n-1) - (a+b)|A||B| \ .
	$$
	
	We know that $x_1, x_2, \ldots, x_{a-1}$ are leaves from which follows 
	$n_{x_i}(x_iu) = 1$ for all $i$ and $n_u(x_iu) = n-1$\,. By a substitution
	we get 
	$$
		\Delta = (a-1)(b-1)(n-1) + (a-1)\sum_{i=1}^{b-1}n_{y_i}(y_iv)n_v(y_iv)
					 + (a+b)(n-1) - a(a+b)(n-a) \ .
	$$
	
	Now we want to argue that $\sum_{i=1}^{b-1}n_{y_i}(y_iv)n_v(y_iv)$ is minimal
	if $y_2, y_3, \ldots, y_{b-1}$ are leaves (the situation is depicted in
	Figure \ref{fig:minimizing}). Assume for contradiction that $|Y_i| > 1$ for some
	$i > 1$.
	\begin{figure}[H]
		\centering
		\includegraphics[width=.3\linewidth]{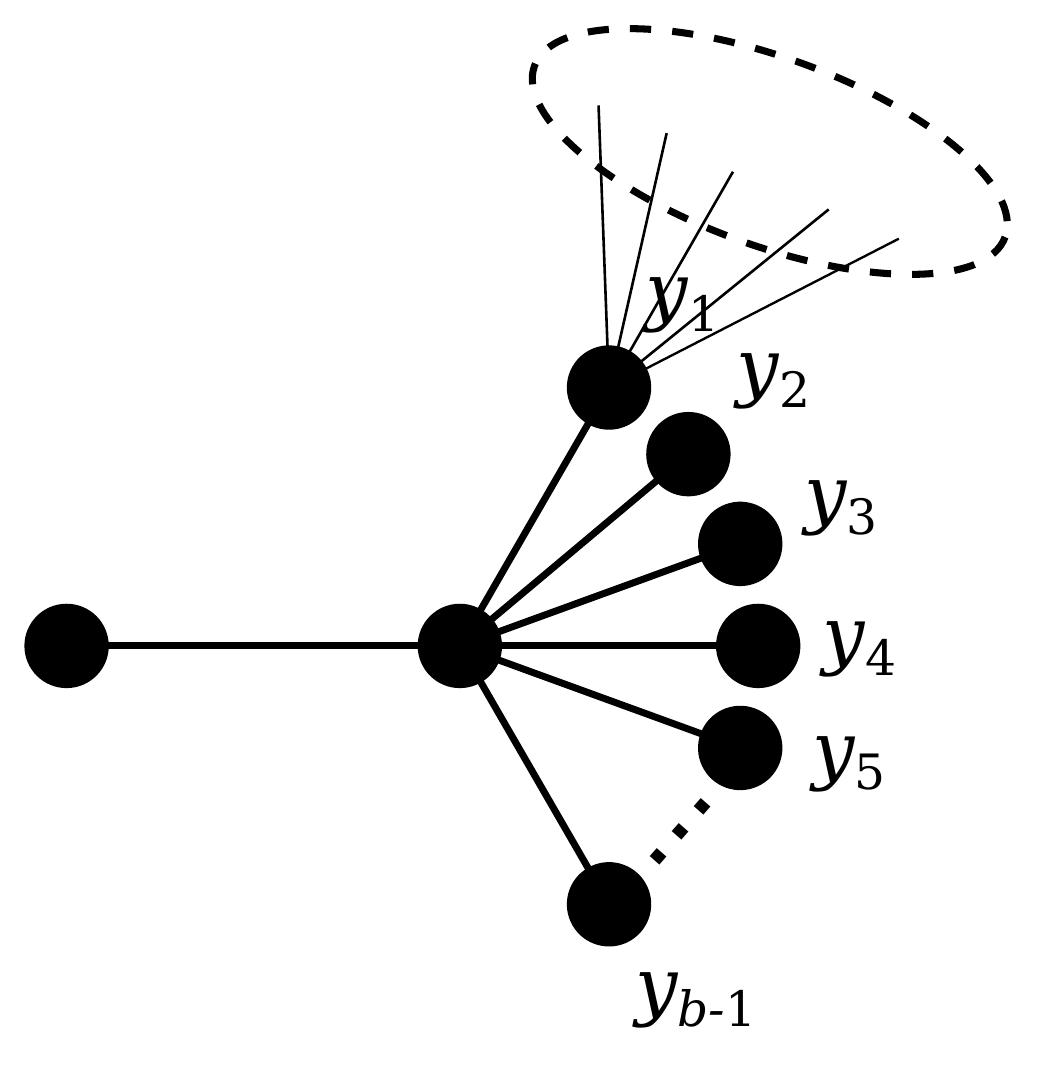}
		\caption{}
		\label{fig:minimizing}
	\end{figure}
	
	Observe that $\sum_{i=1}^{b-1}|Y_i| = n - a - 1$. We know that
	$$
		\sum_{i=1}^{b-1}n_{y_i}(y_iv)n_v(y_iv) = \sum_{i=1}^{b-1} |Y_i| (n - |Y_i|)
		= n\sum_{i=1}^{b-1} |Y_i| - \sum_{i=1}^{b-1} |Y_i|^2 \ .
	$$
	
	Using a well known fact that $(x+1)^2 + (y-1)^2 > x^2 + y^2$ for $x > y$ and
	$x,y \in \mathbb N$\,, we conclude that $\sum_{i=1}^{b-1} |Y_i|^2$ attains the maximum
	value if $|Y_1| = n - a - b + 1$ and $|Y_i| = 1$ for $1 < i \leq b-1$\,.
	
	By using this fact we get that 
	$$
		\sum_{i=1}^{b-1}n_{y_i}(y_iv)n_v(y_iv) \geq (b-2)(n-1) + n(a+b-1)-(a+b-1)^2\,,
	$$ 
	and thus 
	\begin{align*} 
		\Delta & \geq (b-1)(a-1)(n-1) + (a-1) \left( (b-2)(n-1) + n(a+b-1)-(a+b-1)^2 \right)\\
					 & + (a+b)(n-1) - (a+b)a(n-a) \ .
	\end{align*}
	It is enough to show that $\Delta' = 2bn - ab - b^2 + 3a + b - 4n + 2 > 0$\,. For this,
	we distinguish three cases for the value of $b$. For $b=2$ and $b=3$ we obtain respectively
	$$
		\Delta'= 4n - 2a - 4 + 3a + 2 - 4n + 2 > 0
	$$
	and
	$$
		\Delta'= 6n - 3a - 9 + 3a + 3 - 4n + 2 > 0 \ .
	$$
	
	Suppose now that $b \ge 4$. Together with an observation that $bn \ge b(a+b) = ab + b^2$
	we obtain
	$$
		\Delta' \ge bn - 4n + 3a + b + 2 > 0 \ .
	$$
	This completes the proof.
\end{proof} 

\vspace{3mm}

\begin{proposition}\label{prop:bipartite}
	Among all bipartite graphs the balanced complete bipartite graph 
	$K_{\left\lfloor n/2 \right\rfloor, \left\lceil n/2 \right\rceil}$
	achieves the maximum weighted Szeged index.
\end{proposition}

\begin{proof}
	For an edge $uv$ in a bipartite graph it holds that the $n_u + n_v = n$\,.
	Knowing this fact, it is obvious that 
	$$
		\max{[n_u\cdot n_v]} = \left\lfloor\frac{n^2}{4}\right\rfloor
	$$
	Also it is well-known that the maximum number of edges in a bipartite
	graph is $\left\lfloor n^2/4 \right\rfloor$\,, and $\max{[\deg(u) +
	\deg(v)]} = n$\,. Then,
	$$
		\sz(G) \leq n \cdot \left\lfloor\frac{n^2}{4}\right\rfloor ^ 2 \ .
	$$
	Equality holds if and only if $G\cong K_{\left\lfloor n/2 \right\rfloor,
	\left\lceil n/2 \right\rceil}$\,.
\end{proof}

Preliminary computer calculations indicate that the balanced complete bipartite
graph has the maximal weighted Szeged index among all connected graphs. Taking 
into account this and the fact that the same graph is maximal for the ordinary
Szeged index \cite{Dobr97, ChWu09}, we believe that the following conjecture is
true.

\vspace{3mm}

\begin{conjecture}\label{con:maxG}
	For $n$-vertex graph the maximum weighted Szeged index is attained by the
	balanced complete bipartite graph  $K_{\lceil n/2 \rceil, \lfloor n/2 \rfloor}$.
\end{conjecture}

\section{On graphs having minimum weighted Szeged\newline index}

Problem of characterizing graphs with the minimal weighted Szeged index is more
complex than with the graphs that maximizing it. Computer explorations of graphs
having the minimal weighted Szeged index lead to the following conjecture:

\vspace{3mm}

\begin{conjecture}\label{con:minG}
	For $n$-vertex graphs the minimum weighted Szeged index is attained by a tree.
\end{conjecture}

Accepting the Conjecture \ref{con:minG} as true, we were pursuing for the minimal
trees. The results of such computer search are presented in the Table \ref{tab:trees}. 

\begin{center}
	\newcolumntype{M}[1]{>{\centering\arraybackslash}m{#1}}
	\begin{longtable}{rcM{12cm}}
		\caption{Trees from 7 to 25 vertices having minimum weighted Szeged index.}\\
		\label{tab:trees}\\
		\toprule[2pt]
		\multicolumn{1}{c}{\# vertices} & \multicolumn{1}{c}{\,} &
		\multicolumn{1}{c}{Trees with minimal $\sz(G)$ index} \\
		\midrule \\[2mm]
		\endfirsthead

		\multicolumn{3}{c}{\hdashrule{.9\textwidth}{0.3pt}{5mm}}\\
		\multicolumn{1}{c}{\# vertices} & \multicolumn{1}{c}{\,} &
		\multicolumn{1}{c}{Trees with minimal $\sz(G)$} \\
		\midrule \\[2mm]
		\endhead
		
		\multicolumn{3}{c}{Table \ref{tab:trees} continues on the next page} \\
		\midrule
		\endfoot
		
		\bottomrule[2pt]
		\endlastfoot
		
		7  &\,& \includegraphics[width=.6\linewidth]{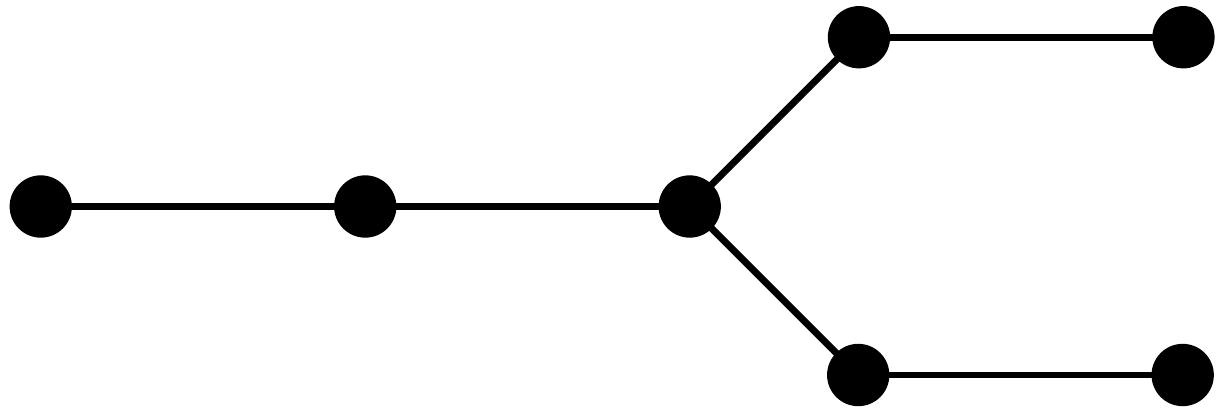}\\[1cm] \midrule
		8  &\,& \includegraphics[width=.6\linewidth]{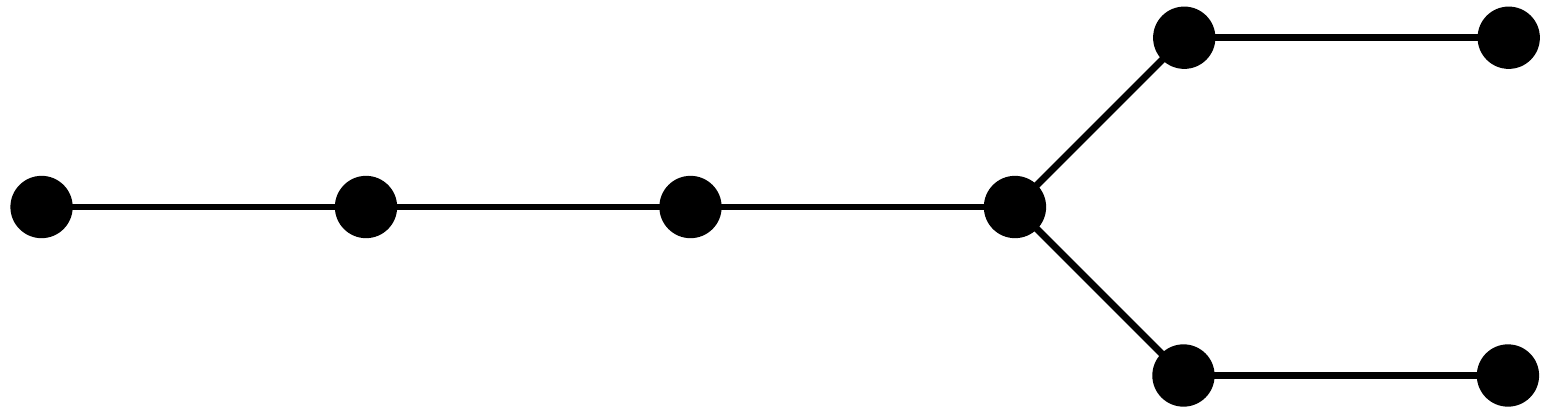}\\[1cm] \midrule
		9  &\,& \includegraphics[width=.4\linewidth]{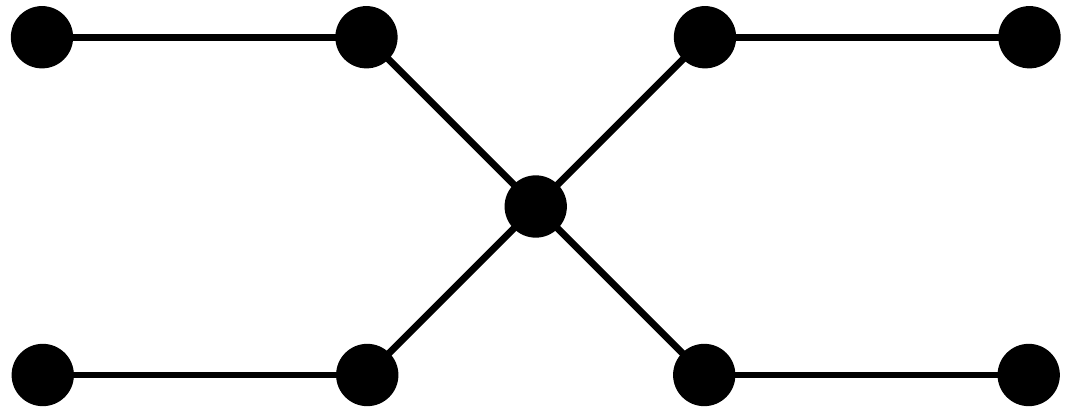}\\[1cm] \midrule
		10  &\,& \includegraphics[width=.6\linewidth]{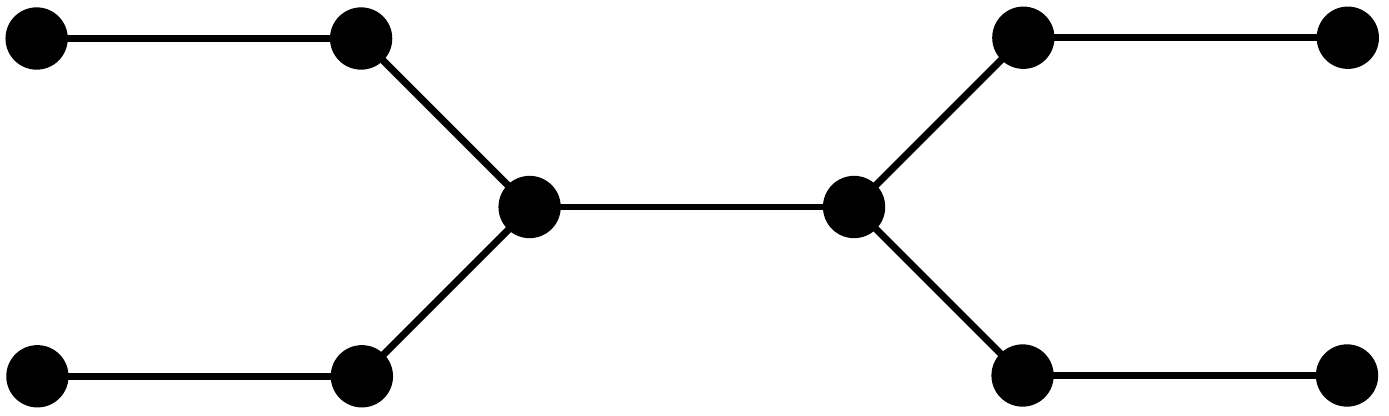}\\[1cm] \midrule
		11  &\,& \includegraphics[width=.6\linewidth]{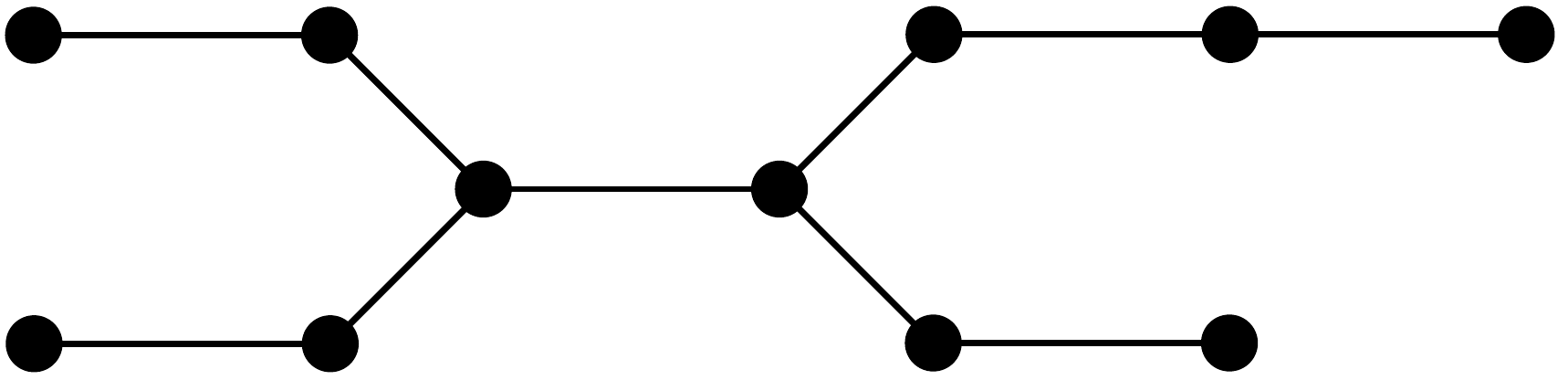}\\[1cm] \midrule
		12  &\,& \includegraphics[width=.6\linewidth]{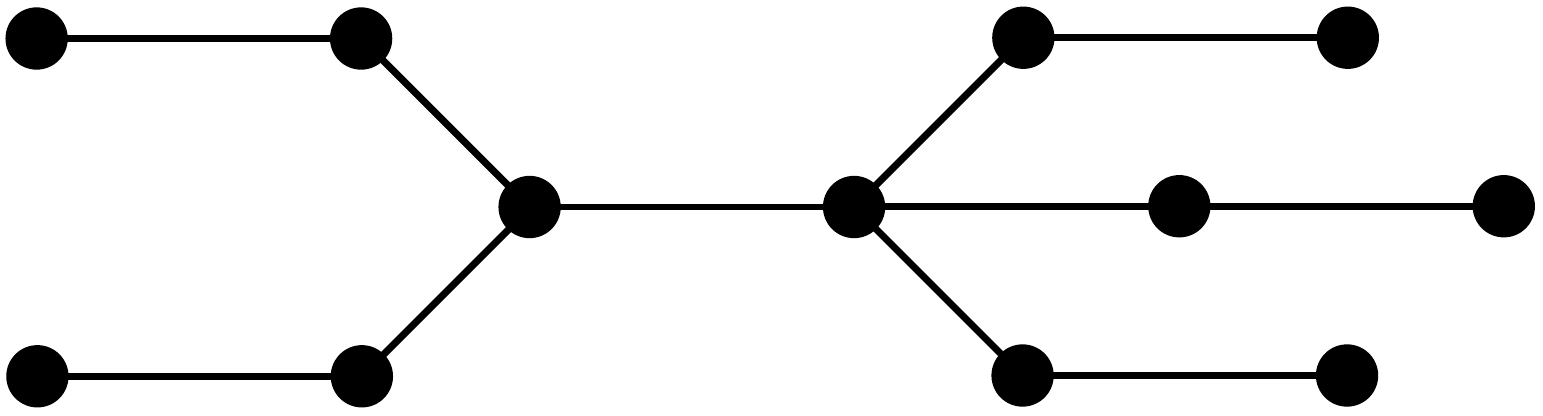}\\[1cm] \midrule
		13  &\,& \includegraphics[width=.6\linewidth]{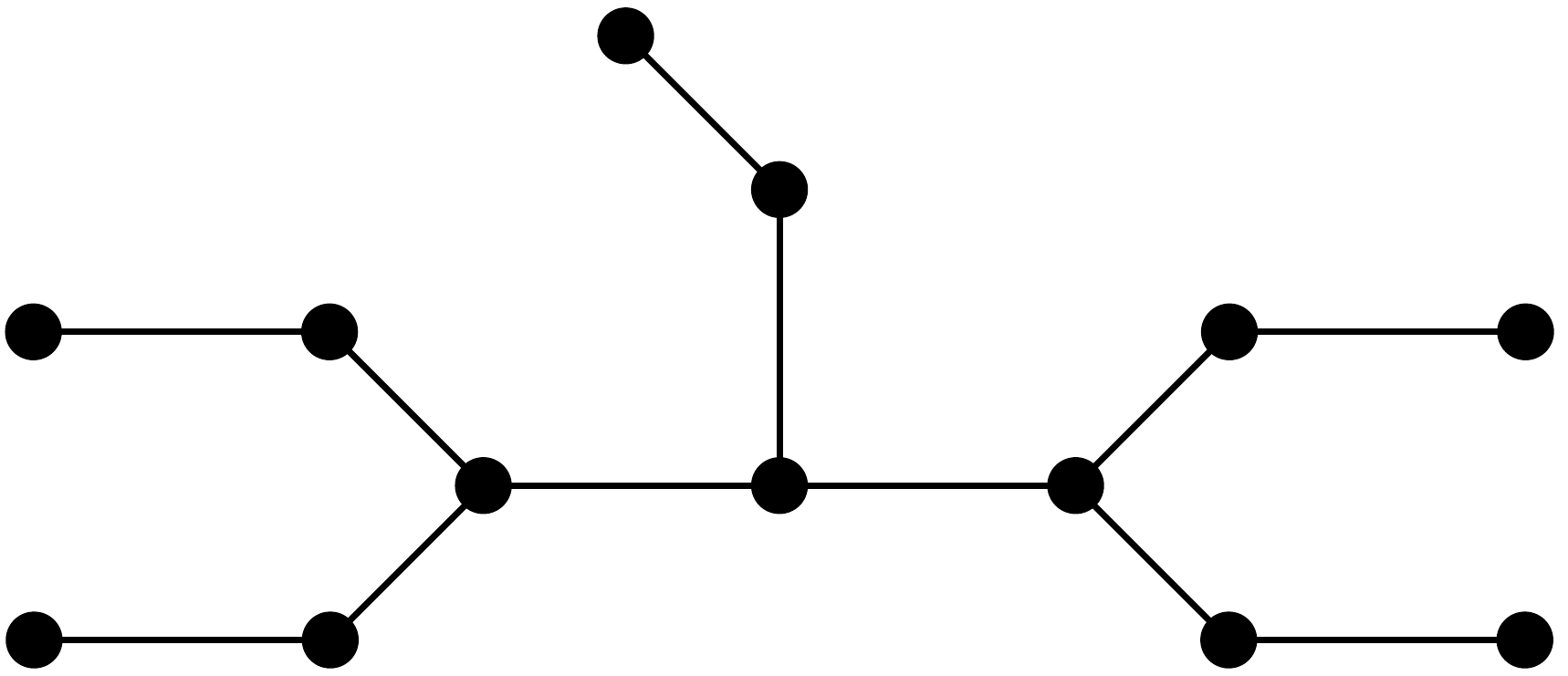}\\[1cm] \midrule
		14  &\,& \includegraphics[width=.6\linewidth]{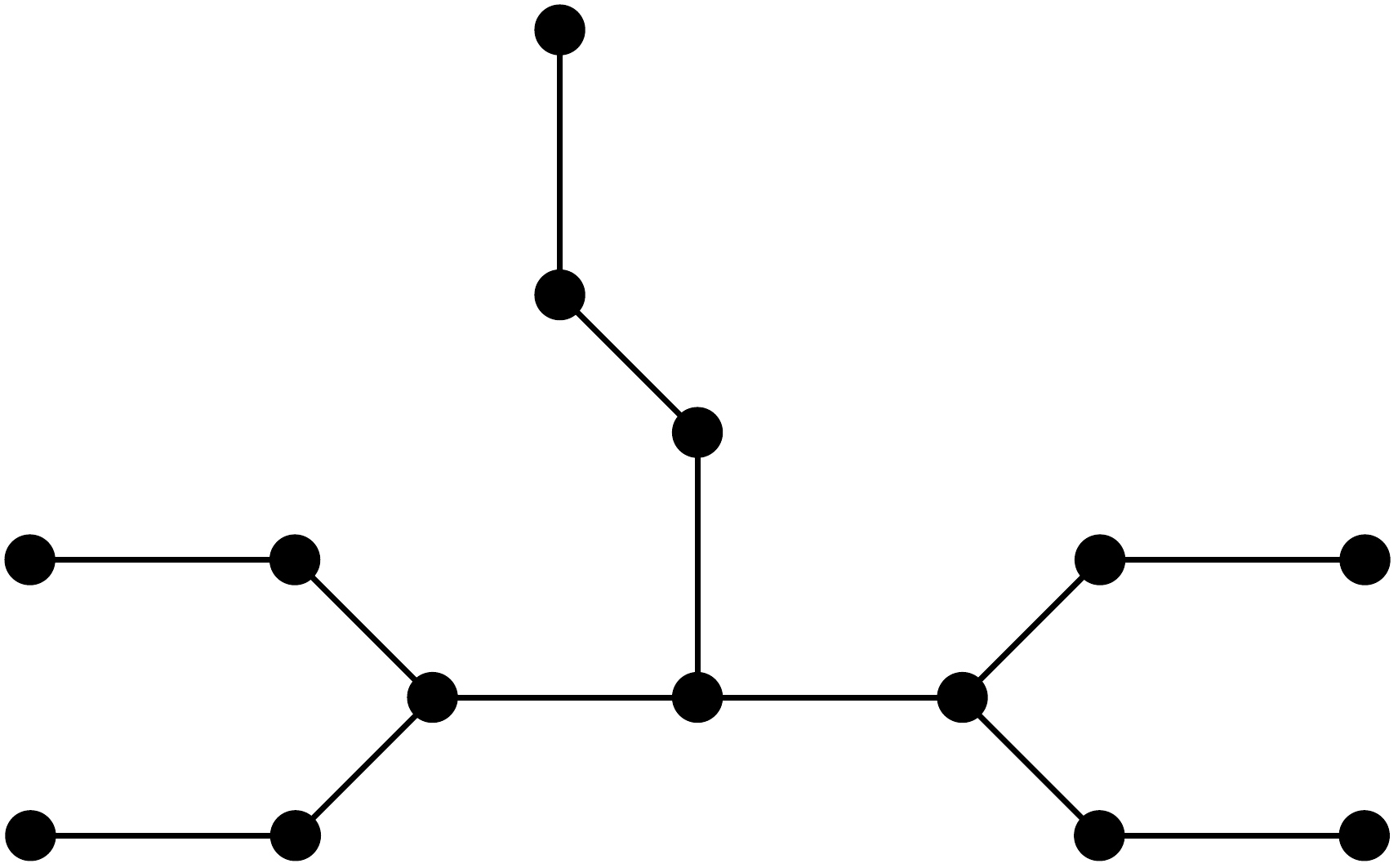}\\[1cm] \midrule
		15  &\,& \includegraphics[width=.6\linewidth]{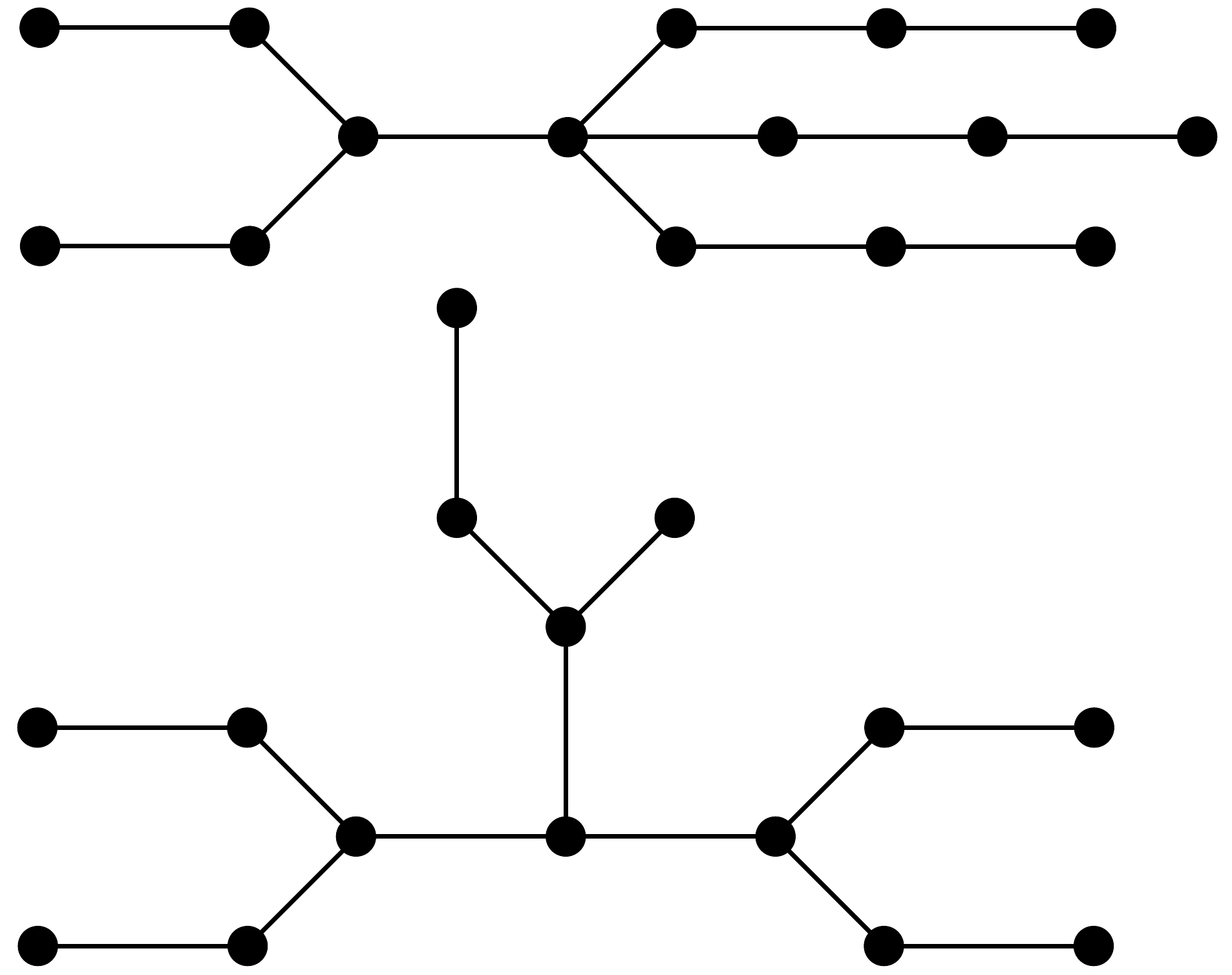}\\[1cm] \midrule
		16  &\,& \includegraphics[width=.6\linewidth]{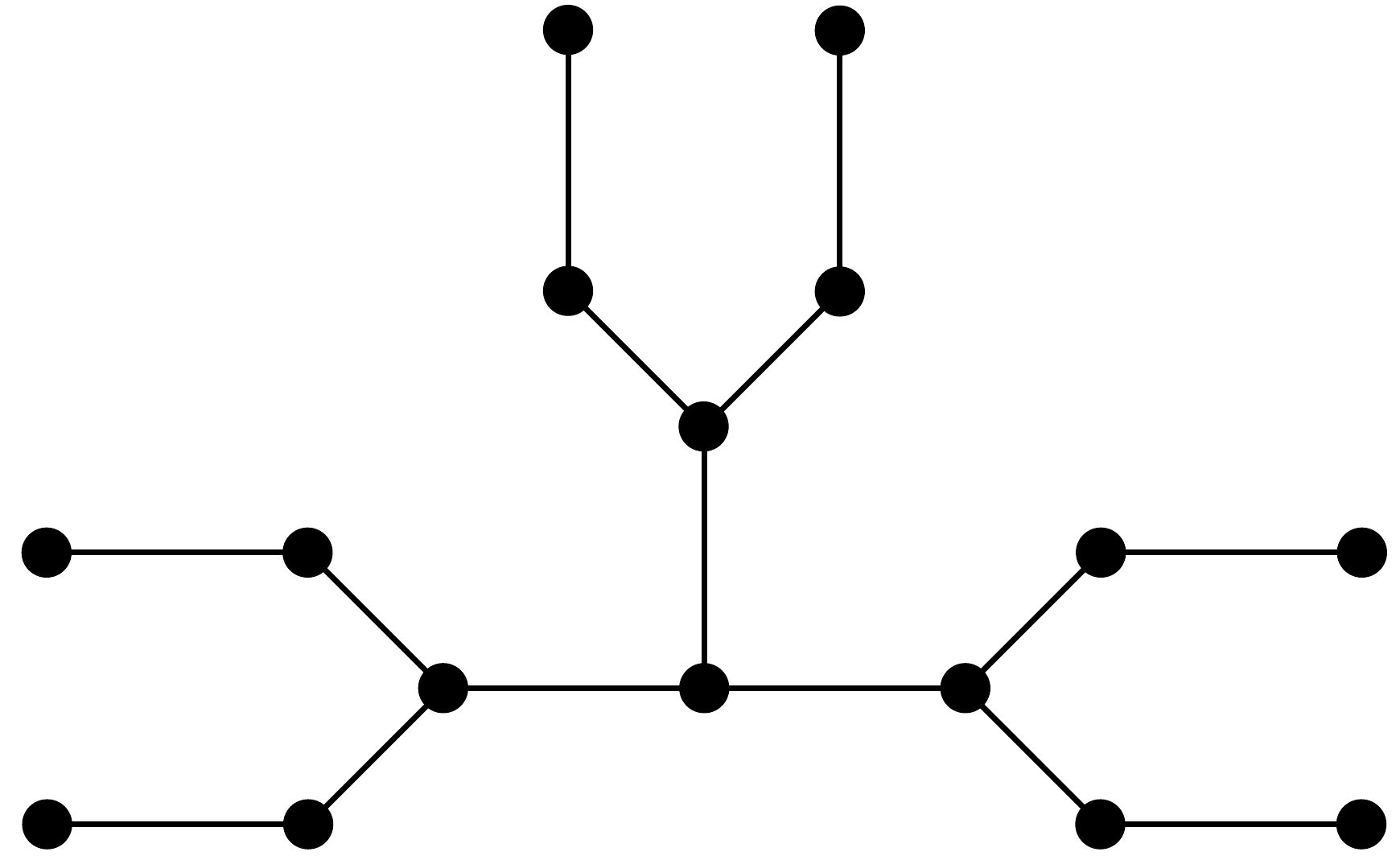}\\[1cm] \midrule
		17  &\,& \includegraphics[width=.6\linewidth]{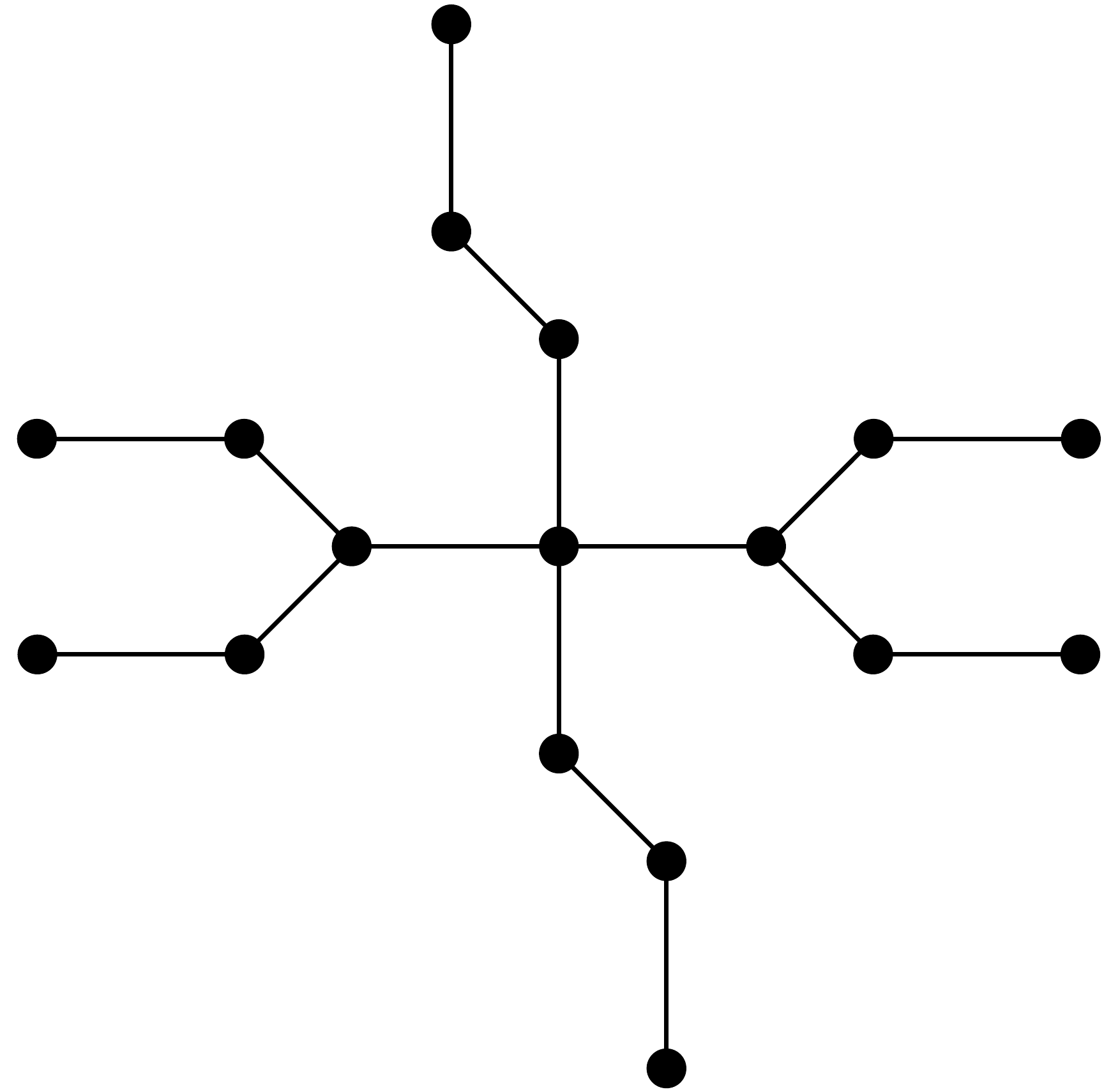}\\[1cm] \midrule
		18  &\,& \includegraphics[width=.95\linewidth]{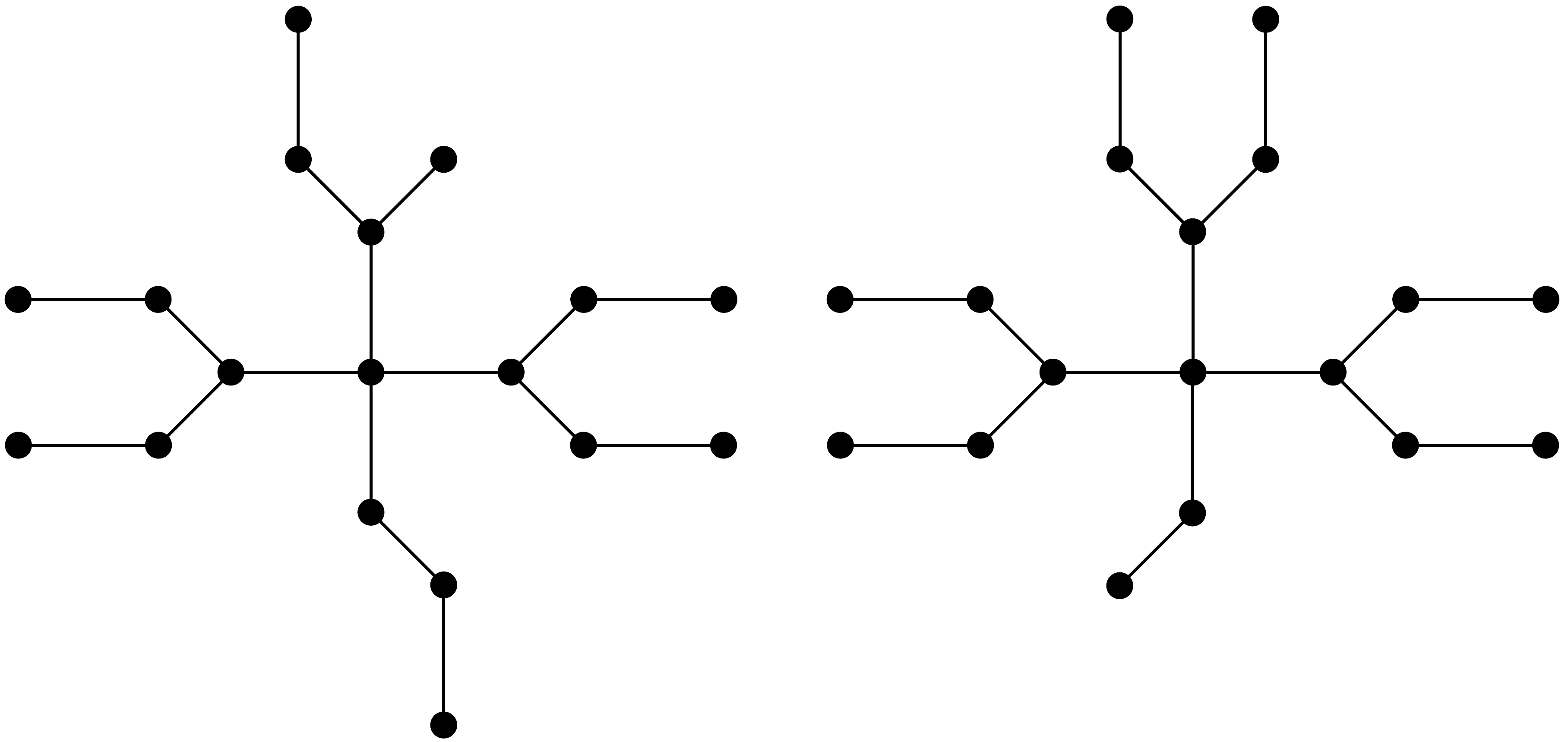}\\[1cm] \midrule
		19  &\,& \includegraphics[width=.6\linewidth]{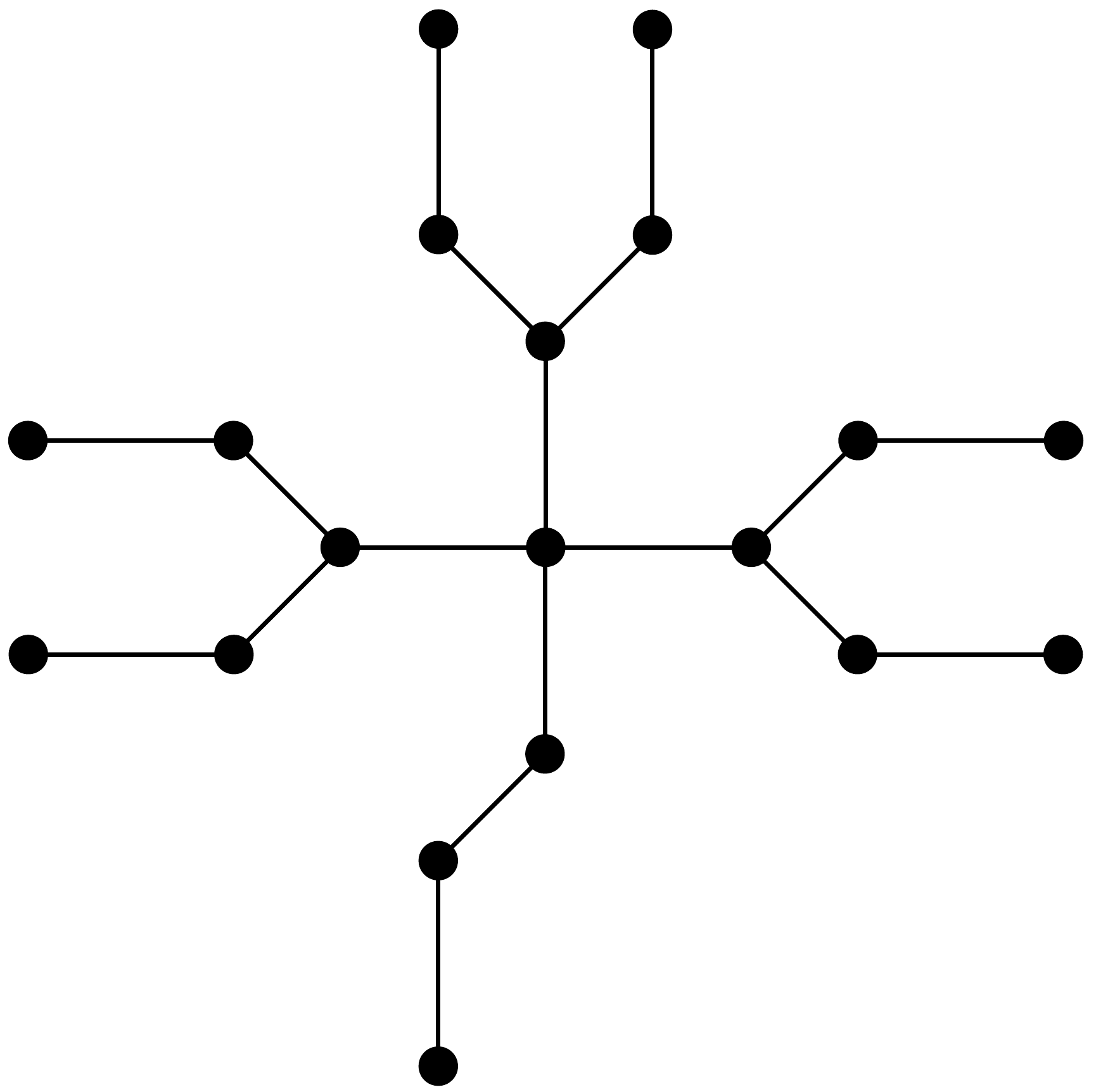}\\[1cm] \midrule
		20  &\,& \includegraphics[width=.6\linewidth]{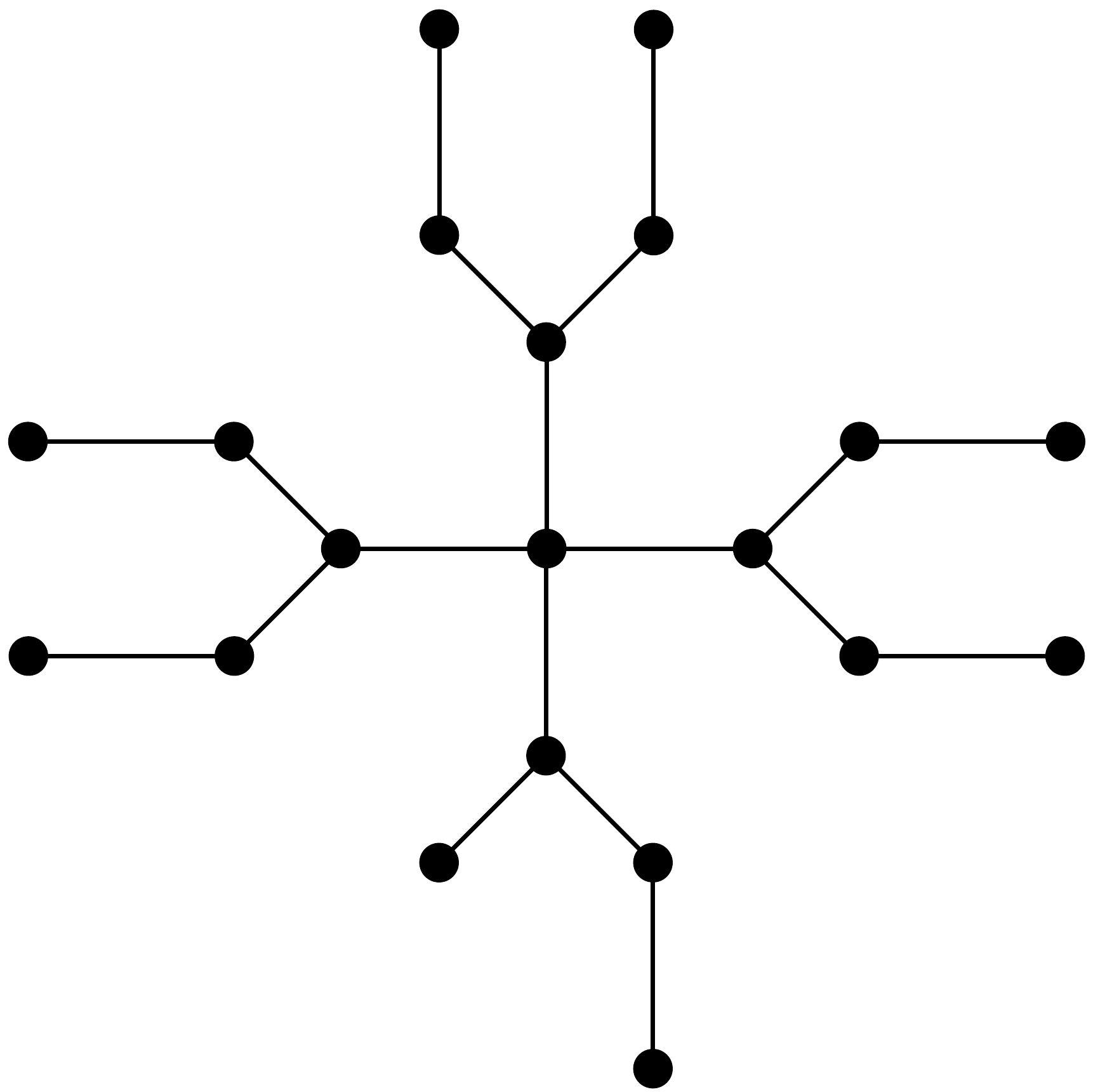}\\[1cm] \midrule
		21  &\,& \includegraphics[width=.6\linewidth]{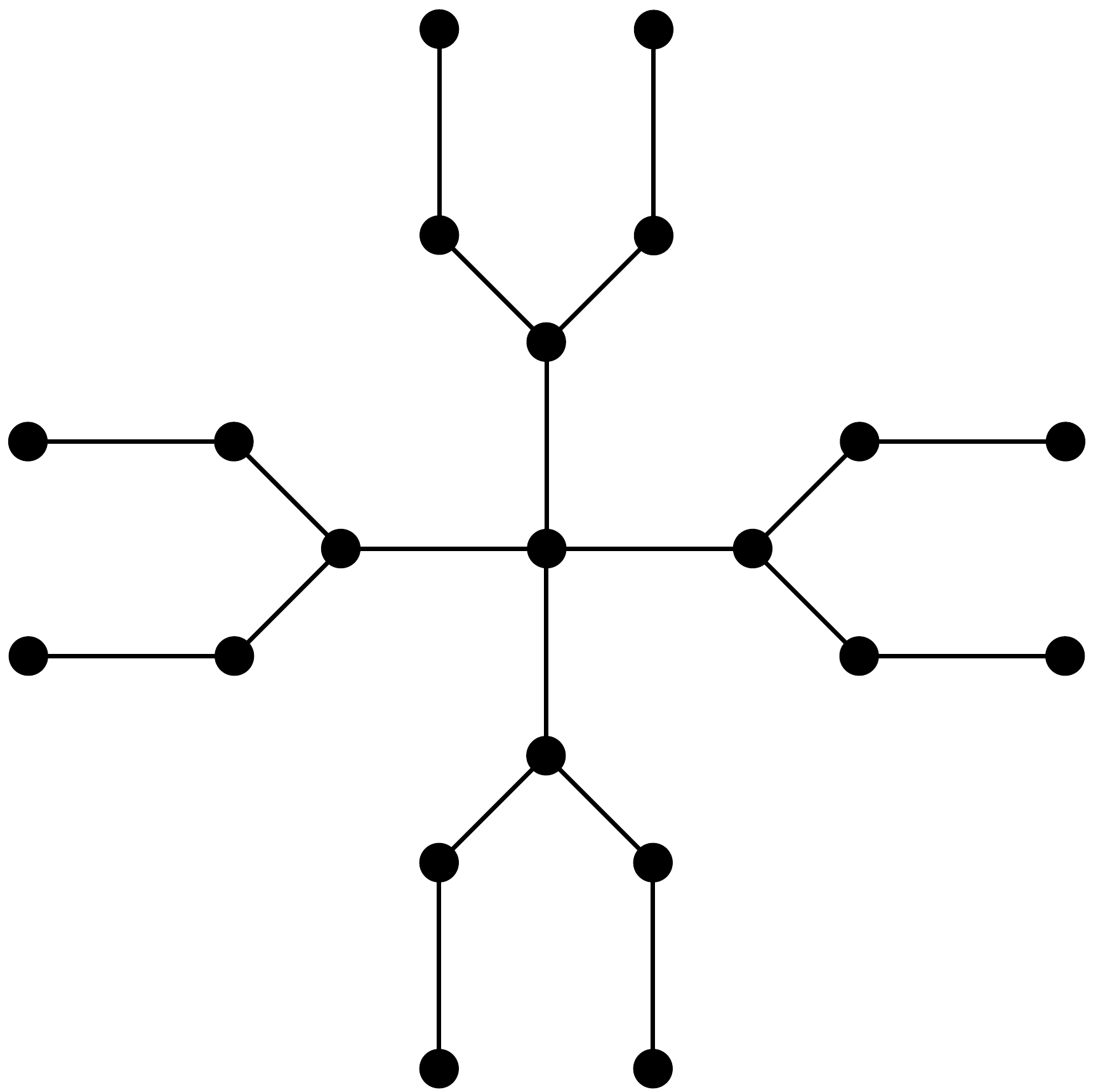}\\[1cm] \midrule
		22  &\,& \includegraphics[width=.6\linewidth]{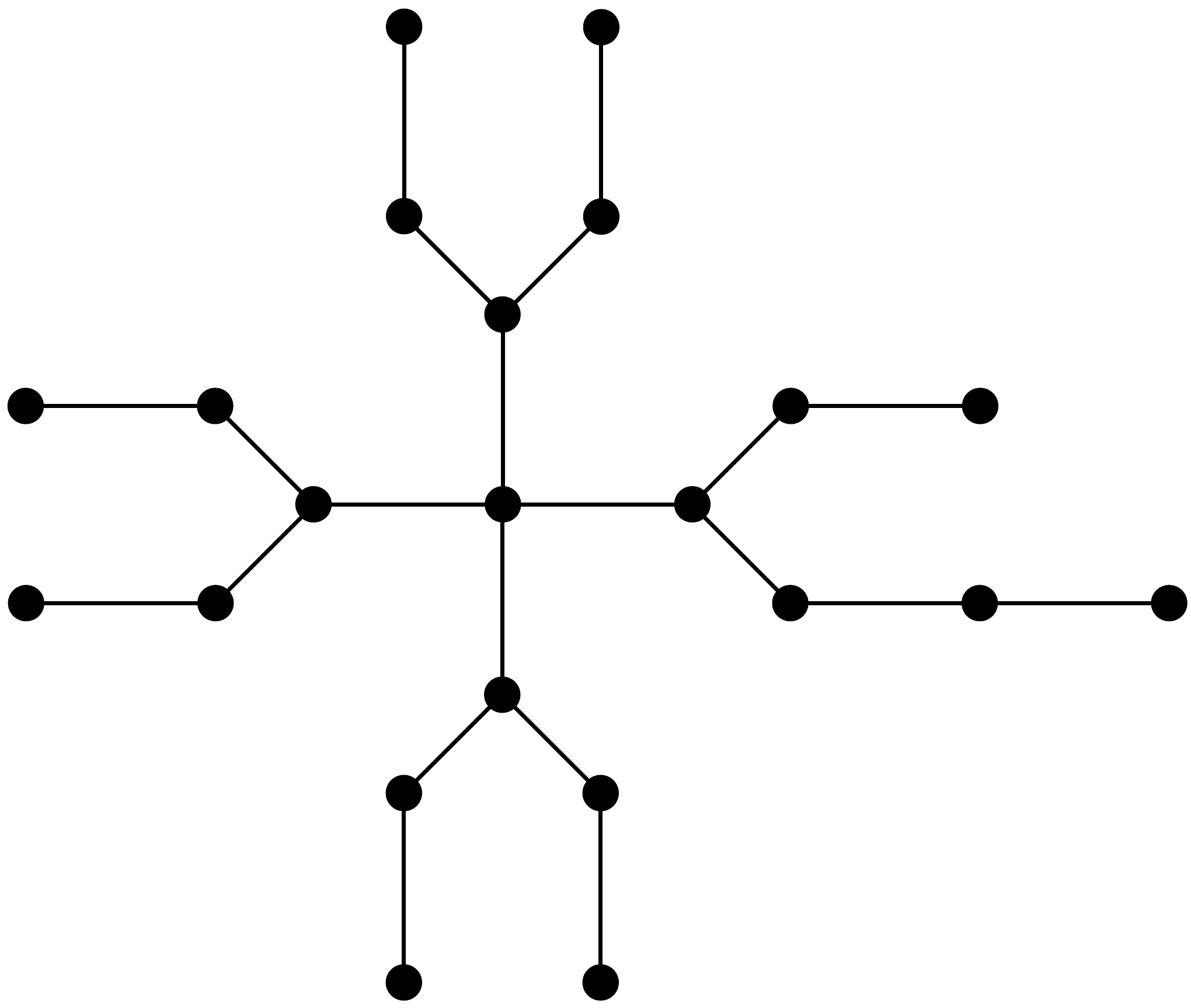}\\[1cm] \midrule
		23  &\,& \includegraphics[width=.6\linewidth]{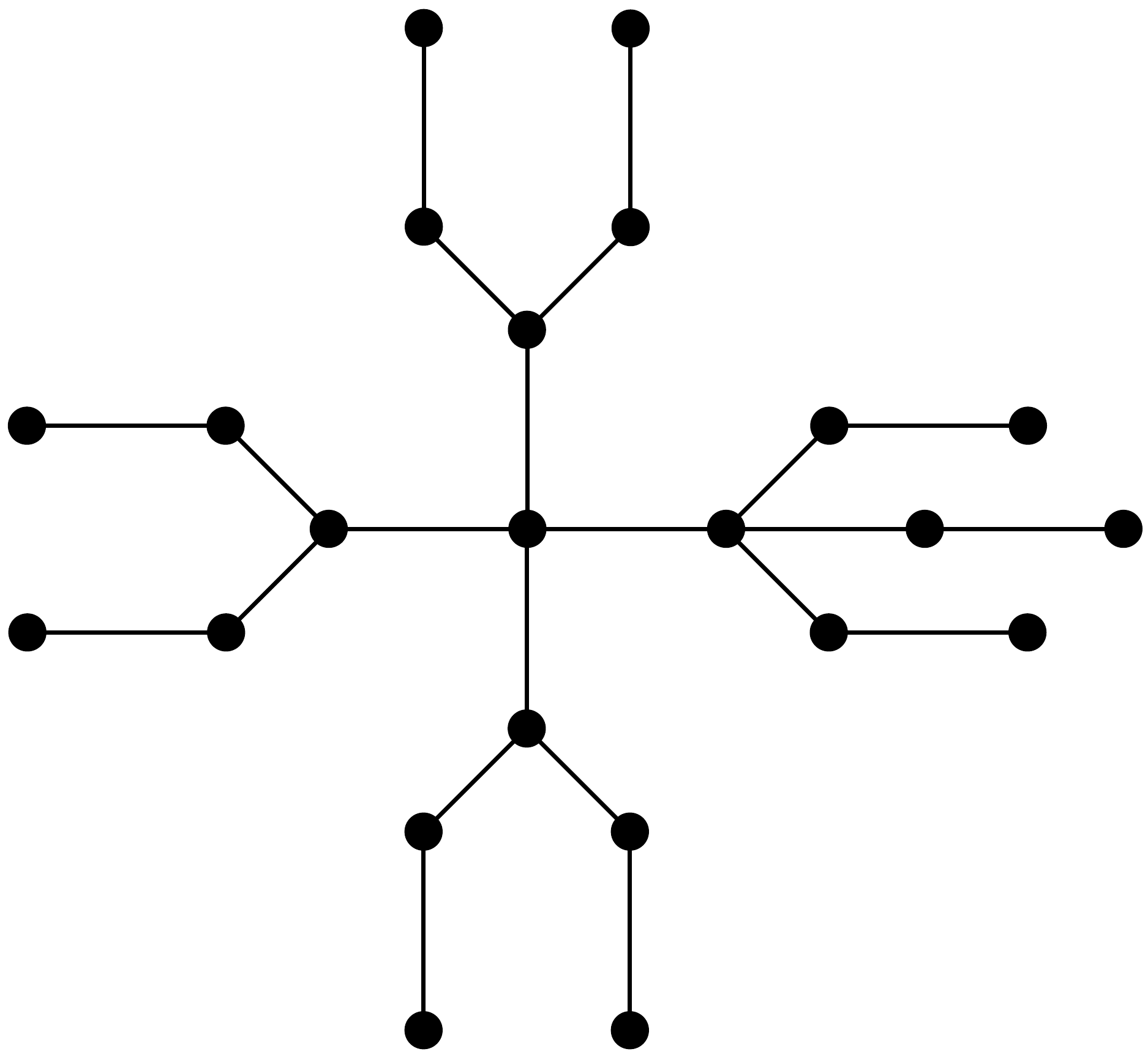}\\[1cm] \midrule
		24  &\,& \includegraphics[width=.6\linewidth]{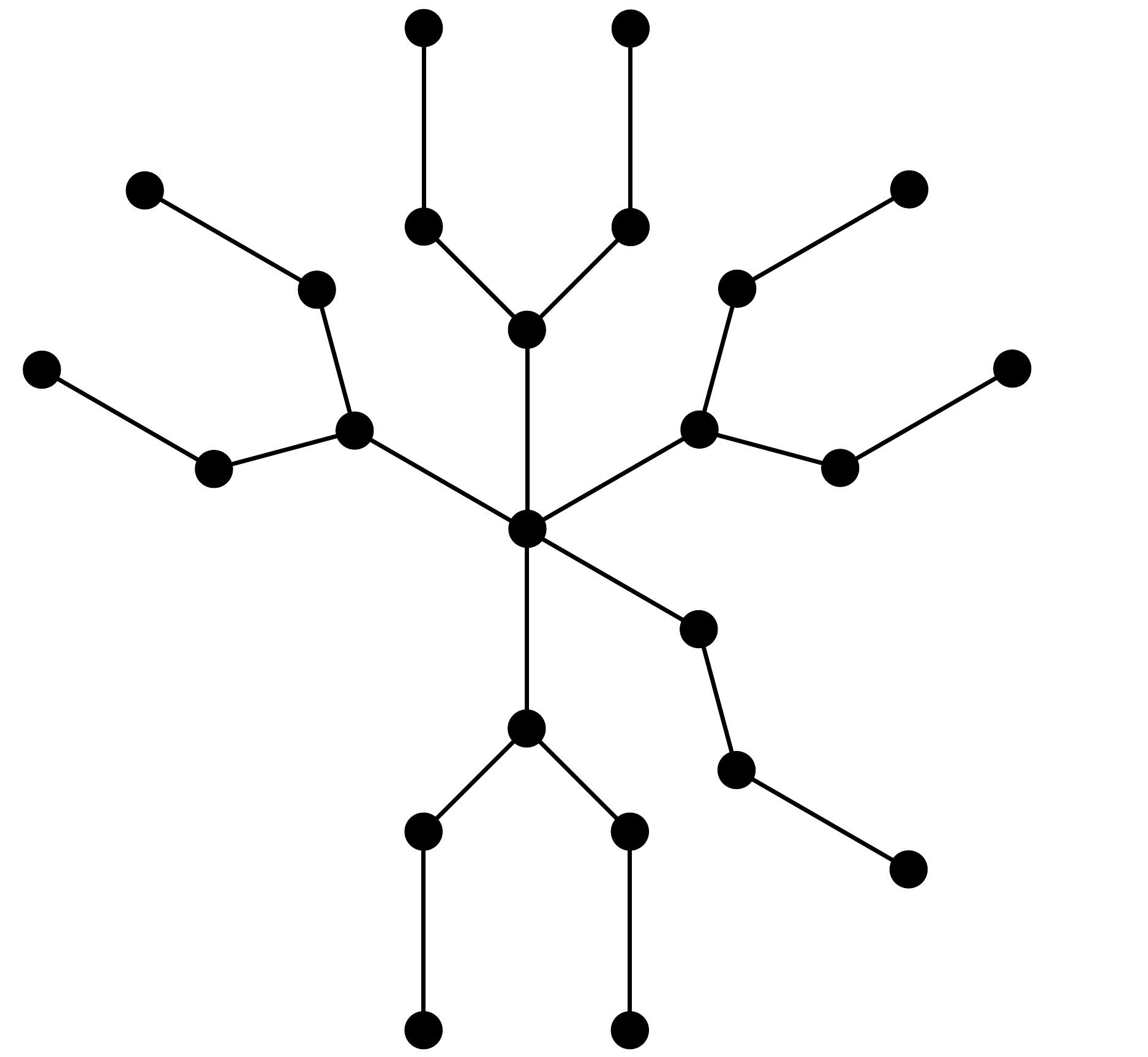}\\[1cm] \midrule
		25  &\,& \includegraphics[width=.6\linewidth]{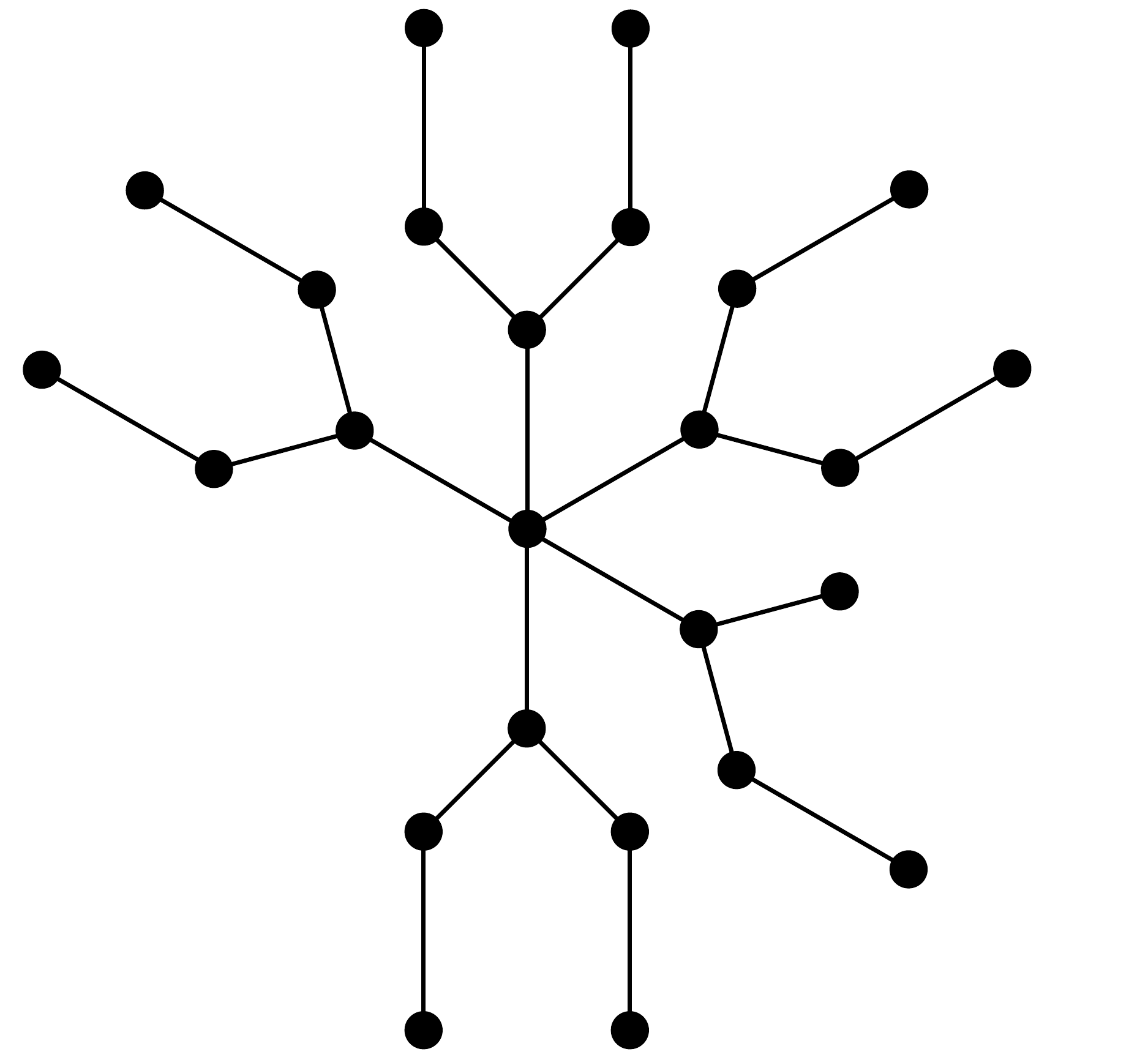}\\[1cm] \midrule
	\end{longtable}
\end{center}

It is peculiar that some of the trees, shown in the Table \ref{tab:trees},
coincide with the minimal trees according to the well-know atom--bond
connectivity index $ABC(T)$ \cite{GuFuIv12}. In particular, minimal trees
for the $\sz(T)$ and $ABC(T)$ for 7--13, 18, and 21 vertices are the same.
It should be noted that apart from these minimal trees there may be other
trees of the same order having either the minimal $ABC(T)$ or the minimal 
$\sz(T)$\,. Minimal trees, shown in the Table \ref{tab:trees}, mostly 
resemble to so-called Kragujevac trees introduced in \cite{HoAhGu14}.

The complete characterization of trees having minimal $\sz(T)$
is beyond our limits at the moment, but we noticed several regularities
from figures in the Table \ref{tab:trees} that will be outlined in the
next subsection.

\subsection{Some properties of trees having minimum weighted Szeged index}

Here, we denote by $T_{\min}$ a tree with the minimum possible weighted
Szeged index on $n$ vertices, and study its properties. Notice that for
some values of $n$, it may not be a uniquely defined tree. 

\vspace{3mm}

\begin{proposition} \label{prop:trans2}
	For $n > 3$, no vertex of degree at least 6 in $T_{\min}$ is adjacent to two leaves.
\end{proposition}

\begin{figure}[H]
\centering
\includegraphics[width=0.4\linewidth]{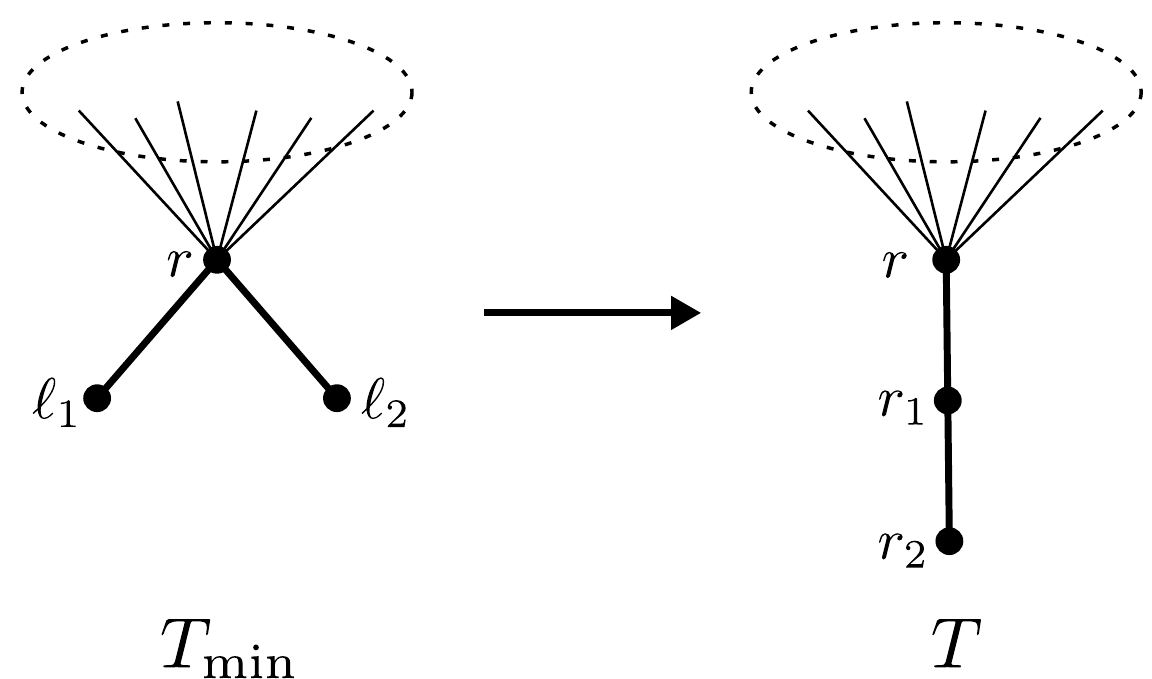}
\caption{An illustration of construction in Proposition \ref{prop:trans2}.}
\label{fig:trans2}
\end{figure}

\begin{proof}
	Suppose that a vertex $r$ of $T_{\min}$ is a vertex of degree $d$ adjacent
	to two leaves $\ell_1$ and $\ell_2$. Denote the vertices 
  $N(v)\setminus\{\ell_1,\ell_2\}$ as $x_1, x_2, \ldots, x_{d-2}$ and the set
  of vertices of the components of $T_{\min}-vx_i$ containing $x_i$ as $X_i$
  for all $i \in \{1, 2, \ldots, d-2\}$.

	Define a new graph $T$ as
	\begin{align*} 
    V(T) &= V(T_{\min}) - \ell_1 - \ell_2 + r_1 + r_2,\\[3mm]
    E(T) &= E(T_{\min}) - r\ell_1 - r\ell_2 + rr_1 + r_1r_2.
	\end{align*}

	It is easy to see that $\deg_{T_{\min}}(r) = \deg_{T}(r) + 1$ and
	$\sum_{i=1}^{d-2}|X_i| = n - 3$. Also, as in the proof of Theorem \ref{thm:maximum},
  the maximum possible value of $\sum_{i=1}^{d-2}|X_i|^2$ is $(d-3) + (n-d)^2$\,.

	The difference $\Delta = \sz(T_{\min}) - \sz(T)$ can then be written as
	\begin{align*} 
    \Delta &= \left(\sum_{i=1}^{d-2}|X_i|\cdot (n-|X_i|)\right) + 2(d+1)(n-1)
           - (d+1)\cdot 2\cdot (n-2) - 3(n-1)\\[3mm]
           &\geq \left(\sum_{i=1}^{d-2}|X_i|\cdot (n-|X_i|)\right) +2d -3n + 5
           = n\cdot (n-3) - \sum_{i=1}^{d-2}|X_i|^2 + 2d -3n + 5\\[3mm]
           &\geq n\cdot (n-3) - (n-d)^2 - (d-3) + 2d -3n + 5
           = 2dn -6n - d^2 + d + 8\ .
	\end{align*}

  Since $d \geq 6$, it holds $dn \geq 6n$. Also, $dn \geq d^2$ and thus we conclude that $\Delta > 0$. This implies that $T$ has smaller weighted Szeged index than $T_{\min}$,
  which is a contradiction. 
\end{proof}

\vspace{3mm}

\begin{proposition} \label{prop:trans1}
	No vertex of degree at least 10 in $T_{\min}$ is simultaneously incident with two 2-rays and a leaf.
\end{proposition}

\begin{figure}[H]
\centering
\includegraphics[width=0.4\linewidth]{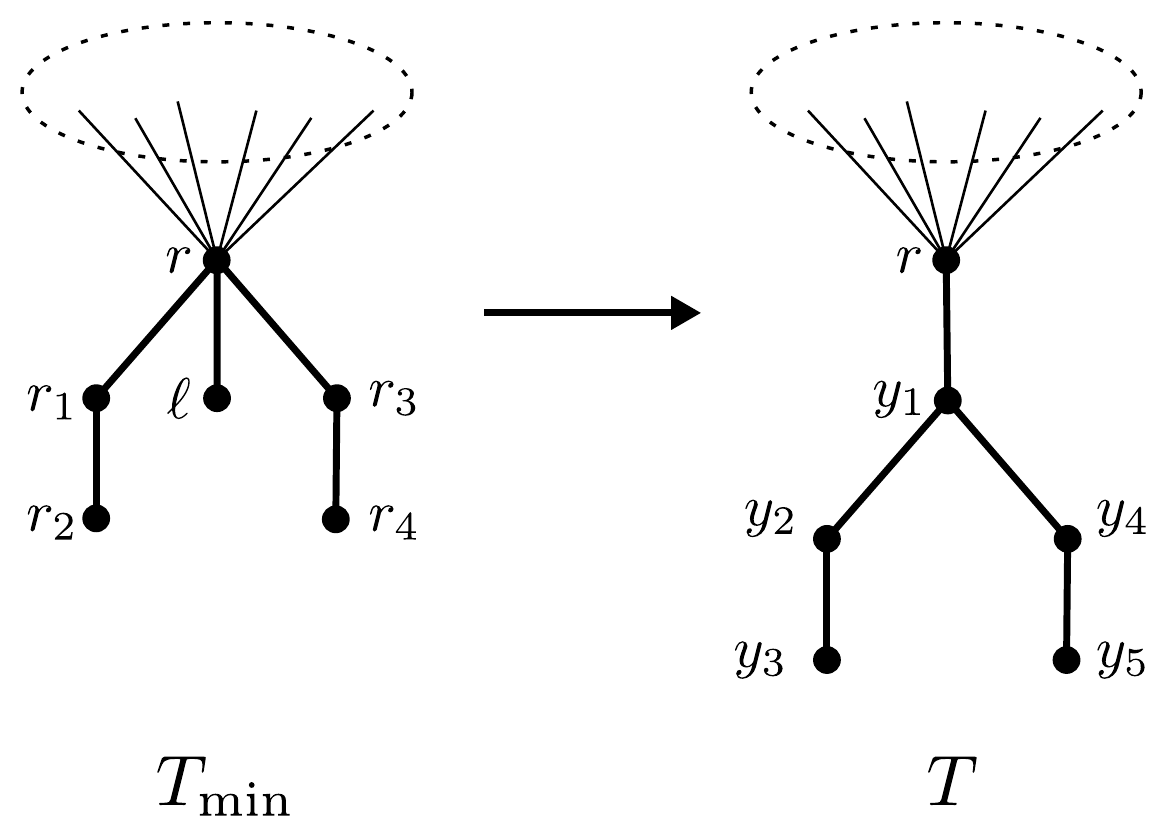}
\caption{An illustration of construction in Proposition \ref{prop:trans1}.}
\label{fig:trans1}
\end{figure}

\begin{proof}
	Suppose contradictory that $T_{\min}$ has a vertex $r$ of degree $d$ incident
  with two 2-rays and a leaf. Denote the vertices $N(v)\setminus\{r_1,r_3,\ell\}$
  as $x_1, x_2, \ldots, x_{d-3}$ and the set of vertices of the components of
  $T_{\min}-rx_i$ containing $x_i$ as $X_i$ for all $i \in \{1, 2, \ldots, d-3\}$\,.

	Define a new graph $T$ as
	\begin{align*} 
    V(T) &= V(T_{\min}) - \{r_1,r_2,r_3,r_4,\ell\} \cup \{y_1, y_2, \ldots, y_5\},\\[3mm]
    E(T) &= \big(E(T_{\min}) \cap E[V(T)]\big) \cup \{ry_1,y_1y_2,y_2y_3,
    y_1y_4,y_4y_5 \}.
	\end{align*}

  Observe that the addition of the edges $r_1r_2$ and $r_3r_4$ to $\sz(T)$ is the
  same as the addition of the edges $y_2y_3$ and $y_4y_5$ to $\sz(T_{\min})$. Also,
  the role of $rr_1$ and $rr_3$ is symmetric and the same holds for $y_1y_2$ and
  $y_1y_4$.

  We now make a few observations, similar to those in Proposition \ref{prop:trans2}.
  We see that $\deg_{T_{\min}}(r) = \deg_{T}(r) + 2$. 
  Furthermore $\sum_{i=1}^{d-3}|X_i| = n - 6$. Also, by the same argument
	as in the proof of Theorem \ref{thm:maximum}, the maximum value of 
  $\sum_{i=1}^{d-3}|X_i|^2$ is $(n-d-2)^2 + d - 4$.

  We claim that $\Delta = \sz(T_{\min}) - \sz(T) > 0$. We can write
  \begin{align*} 
    \Delta &= 2\left(\sum_{i=1}^{d-2}|X_i|\cdot (n-|X_i|)\right) 
           + 2 (2+d)(n-2)2 + (d+1)(n-1) - (d+1)(n-5)5\\[3mm] 
           &- 20(n-2) = 2\sum_{i=1}^{d-3}|X_i| 
           - 2\sum_{i=1}^{d-3}|X_i|^2 + 2 (2+d)(n-2)2 + (d+1)(n-1)\\[3mm]
           &- (d+1)(n-5)5 - 20(n-2) \geq -2d^2 + 4dn + 6d - 20n + 48\ .
  \end{align*}

  Since $d \ge 10$ by our assumption, it holds that $2dn \geq 20n$ and also
  $2dn \geq 2d^2$. Thus, $\Delta > 0$ and we conclude that $T$ has smaller weighted
  Szeged index than $T_{\min}$, which is a contradiction.
\end{proof}

\vspace{3mm}

\begin{proposition} \label{prop:trans3}
	No vertex in $T_{\min}$ is incident to a 4-ray.
\end{proposition}

\begin{figure}[H]
\centering
\includegraphics[width=0.4\linewidth]{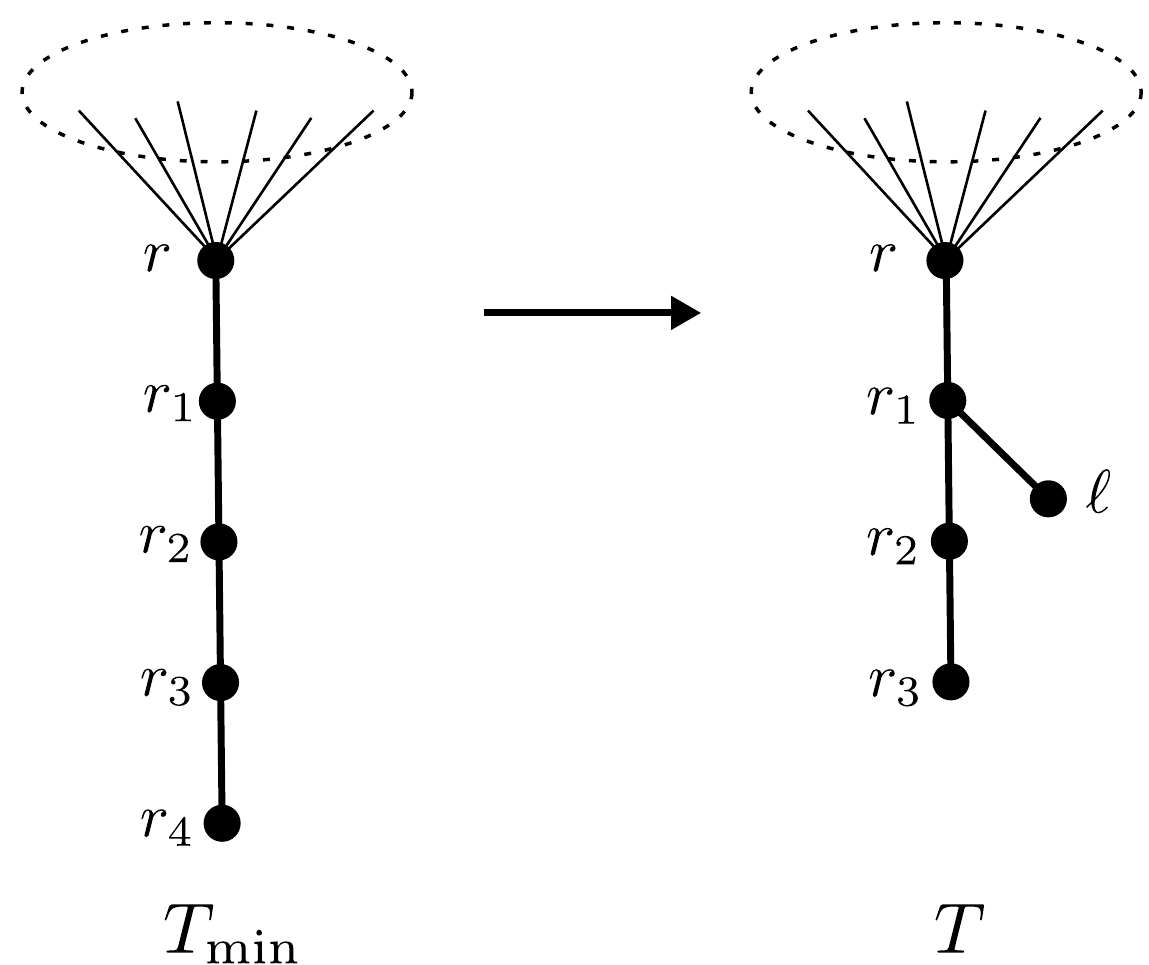}
\caption{An illustration of construction in Proposition \ref{prop:trans3}.}
\label{fig:trans3}
\end{figure}

\begin{proof}
  For $n\le 6$, our computer search verified the claim. So assume that 
  $n \ge 7$. Suppose that we have a vertex $r$ of degree $d$ in $T_{\min}$ incident
  to a 4-ray vertices $rr_1r_2r_3r_4$, as is shown in Figure \ref{fig:trans3}.
  We define a new tree $T$ in the following way.
  \begin{align*} 
    V(T) &= V(T_{\min}) - r_4 + \ell,\\
    E(T) &= E(T_{\min}) - r_3r_4 + r_1\ell \ .
  \end{align*}
	
  We define $\Delta = \sz(T_{\min}) - \sz(T)$ and our aim is to show that
  $\Delta > 0$\,. By the definition of weighted Szeged index	it suffices
  to calculate differences on the edges $\{r_1,r_2,r_3,r_4,r,\ell\}$. Also,
  the summand for the edge $r_3r_4$ in $\sz(T_{\min})$ is equal to the 
  summand of the edge $r_2r_3$ in $\sz(T)$\,. Thus,
  \begin{align*} 
    \Delta &= \left((d+2)\cdot4\cdot(n-4) + 4\cdot 3\cdot (n-3)
           + 4\cdot 2\cdot (n-2)\right) \\[3mm]
           &+\left(-(d+3)\cdot 4 \cdot (n-4) - 4 \cdot (n-1)
           - 5 \cdot 2 \cdot (n-2)\right) \\[3mm]
           &= (4dn - 16d + 28n - 84) + (-4dn + 16d + 26n + 72) = 2n - 12.
\end{align*}
Thus, $\Delta > 0$ for all $n \ge 7$. Therefore, $T_{\min}$ is not a tree with the minimum weighted
Szeged index among $n$-vertex trees, a contradiction.
\end{proof}

\vspace{3mm}

\begin{proposition}\label{prop:trans4}
  No vertex of degree 3 in  $T_{\min}$ is adjacent to two 3-rays.
\end{proposition}

\begin{figure}[H]
\centering
\includegraphics[width=0.4\linewidth]{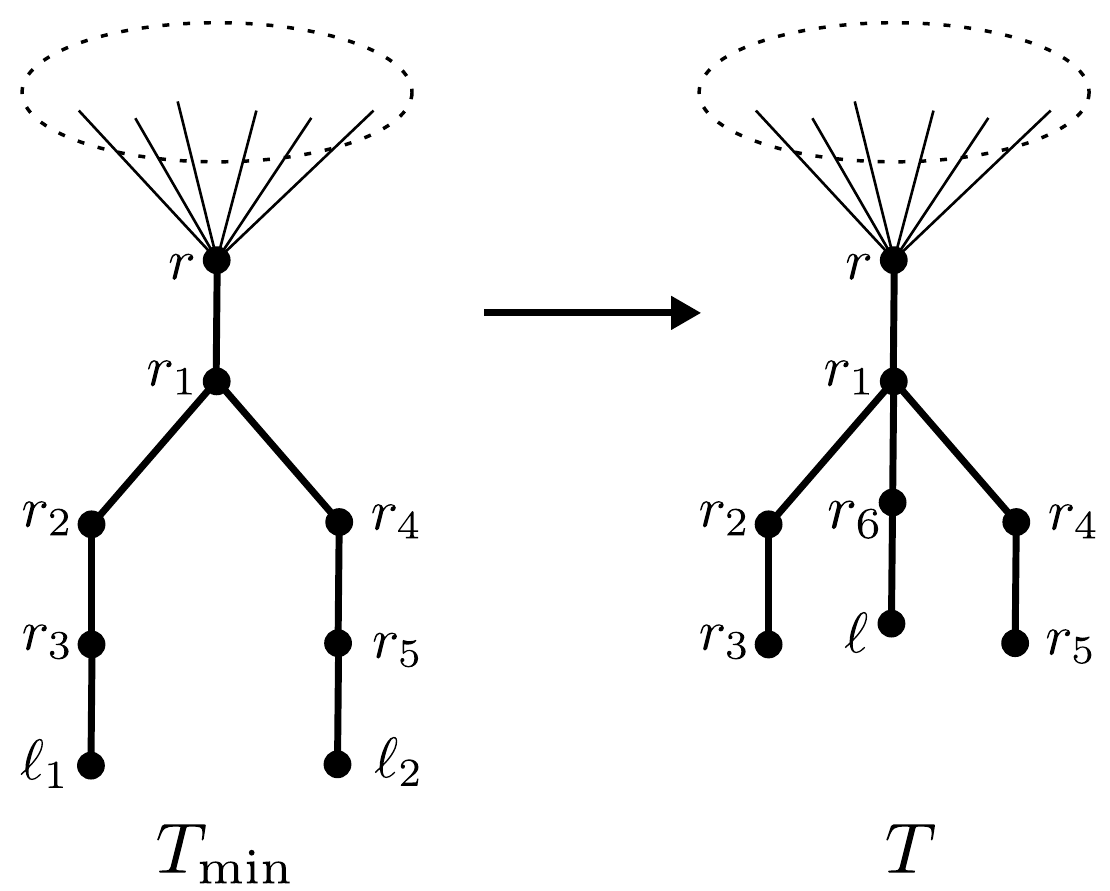}
\caption{An illustration of construction in Proposition \ref{prop:trans4}.}
\label{fig:trans4}
\end{figure}

\begin{proof}
  Suppose that the tree $T_{\min}$\,, shown in Figure \ref{fig:trans4}, is 
  a tree with minimal $\sz(T)$\,.
  Let us transform $T_{\min}$ to $T$ in a way
  as it is illustrated in the Figure \ref{fig:trans4}. In particular,
  \begin{align*}
    V(T) &= V(T_{\min}) - \ell_1 - \ell_2 + r_6 + \ell \\[3mm]
    E(T) &= E(T_{\min}) - r_3\ell_1 - r_5\ell_2 + r_1r_6 + r_6\ell \ .
  \end{align*}
  
  We define $\Delta = \sz(T_{\min}) - \sz(T)$\,. If we show that $\Delta > 0$
  will prove that transformation, depicted in Figure \ref{fig:trans4}, is
  minimizing the $\sz(T)$\,. It is obvious that $\Delta$ depends only on edges
  that are existing among vertices $\{r, r_1, r_2, r_3, r_4, r_5, r_6, \ell,
  \ell_1,  \ell_2\}$\,. Also, contributions of edges $r_3\ell_1$ and $r_5\ell_2$
  in the $T_{\min}$ is equal to the contributions of edges $r_2r_3$ and 
  $r_4r_5$ in the $T$\,. We label the degree of the vertex $r$ in $T_{\min}$
  by $d$ like in previous propositions. Thus,
  \begin{align*}
    \Delta &= 2\cdot\left[ (2+2)\cdot 2\cdot (n-2) + (2+3)\cdot 3\cdot (n-3)\right]
           + (d + 3)\cdot 7\cdot (n-7)\\[3mm]
           &- 3\cdot\left[(2 + 4)\cdot 2\cdot (n-2)\right] - (1+2)\cdot(n-1)
           - (d + 4)\cdot 7\cdot (n-7)\\[3mm]
           &= 30(n-3) - 20(n-2) - 3(n-1) - 7(n-7) = 2 > 0 \ .
  \end{align*}
  
  We obtained that for all $n > 7$ the $\Delta > 0$\,. This indicates that our
  assumption that $T_{\min}$ is a tree with minimum $\sz(T)$ is false, i.e. the
  transformation given in Figure \ref{fig:trans4} is minimizing the weighted
  Szeged index.
\end{proof}

\section{Conclusion}

We have been examining the extremal graphs with respect to weighted Szeged index.
It has been proven that the star graph has the maximum value of the $\sz(T)$ among
all trees. In addition, among all bipartite graphs, the balanced complete bipartite
graph achieves the maximum value of the weighted Szeged index. We did not succeed to
characterize graph (or graphs) that has the maximum $\sz(G)$ among all connected
graphs. Based on a computer investigation, it is believed that it would be the
balanced complete bipartite graph (see Conjecture \ref{con:maxG}).

Characterizing a graph (or graphs) that are minimizing weighted Szeged index is
more complex than finding those with maximum value. We assume that among all
connected graphs, it must be a tree (see Conjecture \ref{con:minG}). Then, the
minimal trees with respect to $\sz(T)$ obtained by computer and some of their 
properties are presented.

Proving Conjecture \ref{con:maxG} and \ref{con:minG} and characterizing corresponding
graphs remains a task for future research endeavors.

\baselineskip=.24in

\vspace{8mm}

\noindent{\it Acknowledgments\/}: The contribution was partially supported by ARRS Project P1-0383,
bilateral project BI-RS/18-19-026, and Ministry of Education, Science and
Technological Development of Republic of Serbia through the grant No. 174033.
The first and the third author would like to acknowledge the support
of the grant SVV-2017-260452.
The first author was supported by the Charles University Grant Agency, project
GA UK 1158216.

\end{document}